\documentclass[11pt,dvips,twoside,letterpaper]{article}
\usepackage{pslatex}
\usepackage{fancyhdr}
\usepackage{graphicx}
\usepackage{geometry}
 \RequirePackage[T1]{fontenc}

\def\figurename{Figure} 
\makeatletter
\renewcommand{\fnum@figure}[1]{\figurename~\thefigure.}
\makeatother

\def\tablename{Table} 
\makeatletter
\renewcommand{\fnum@table}[1]{\tablename~\thetable.}
\makeatother

\usepackage{amsmath}
\usepackage{amssymb}
\usepackage{amsfonts}
\usepackage{amsthm,amscd}

\newtheorem{theorem}{Theorem}[section]
\newtheorem{lemma}[theorem]{Lemma}

\theoremstyle{definition}
\newtheorem{definition}[theorem]{Definition}

\theoremstyle{remark}
\newtheorem{remark}[theorem]{Remark}

\numberwithin{equation}{section}

\def\P{\mathbb P}

\def\R{\mathbb R}
\def\E{\mathbb E}

\def\Q{\mathbb Q}

\def\E{\mathbb E}

\begin{document}
\title{{Reflected backward doubly stochastic differential
equations with discontinuous generator}}
\author{Auguste Aman\thanks{augusteaman5@yahoo.fr, corresponding author}\, and Jean Marc Owo \thanks{owo$_{-}$jm@yahoo.fr}\\
{\it U.F.R Mathématiques et informatique, Universit\'{e} de Cocody},\\ {\it 582 Abidjan 22, C\^{o}te d'Ivoire}}

\date{}
\maketitle

\begin{abstract}
In this note, we study one-dimensional reflected backward doubly
stochastic differential equations (RBDSDEs) with one continuous barrier and discontinuous generator (left-or right-continuous).
By a comparison theorem establish here for RBDSDEs, we provide a minimal or a maximal solution to RBDSDEs.
\end{abstract}


\vspace{.08in} \noindent \textbf{Keywords}: Reflected backward doubly stochastic differential equation, discontinuous generator, comparison theorem.

\section{Introduction}
Backward doubly stochastic differential equations (BDSDEs for
short) are equations with two different directions of stochastic
integrals, i.e., the equations involve both a standard (forward)
stochastic integral $dW_{t}$ and a backward stochastic integral
$\overleftarrow{dB}_{t}$: for $t\in[0,T],$
\begin{eqnarray}\label{eq0}
Y_t=\xi+\int_{t}^{T}f(s,Y_s,Z_s)ds+
\int_{t}^{T}g(s,Y_s,Z_s)\overleftarrow{dB}_{s}-\int_{t}^{T}Z_sdW_{s},
\end{eqnarray}
where $\xi$ is a random variable termed the terminal condition, $f:\Omega \times [0,T]\times \mathbb{R}^k \times
\mathbb{R}^{d}\rightarrow \mathbb{R}$, $g:\Omega \times
[0,T]\times \mathbb{R}^k \times
\mathbb{R}^{d}\rightarrow \mathbb{R}^{l}$ are two jointly measurable processes, $W$ and $B$ are two mutually
independent standard Brownian motion, with values, respectively in
$\mathbb{R}^{d}$ and $\mathbb{R}^{l}$.
This kind of equations has been introduced by Pardoux and Peng
\cite{PardPeng} in $1994$. A solution of that equation is a couple of jointly measurable processes
$(Y,Z)$ with values in $\mathbb{R}^k \times
\mathbb{R}^{d}$ which mainly satisfies Eq. \eqref{eq0}. The authors have proved an existence and unique solution when $f$ and $g$ are uniform Lipschitz. They also
showed that BDSDEs can produce a probabilistic representation for solutions to some quasi-linear stochastic partial differential
equations. Unfortunately, the uniform Lipschitz condition cannot be satisfied in many applications.
Many authors have attempted to relax this condition on the coefficients $f$ and $g$. For instance, 
Shi et al. \cite{Shal} weakened the uniform Lipschitz assumptions to linear growth and continuous conditions by virtue of the comparison theorem introduced by themselves. They obtain the existence of solutions to Eq. \eqref{eq0} but without uniqueness. Recently, N'zi and Owo \cite{NO} have proved an existence solution to Eq. \eqref{eq0} when $f$ is discontinuous in $y$ and continuous in $z$.

In this note, we study the now well-know reflected backward doubly stochastic differential equations (RBDSDEs for
short):
\begin{eqnarray}
Y_t=\xi+\int_{t}^{T}f(s,Y_s,Z_s)ds+\int_{t}^{T}g(s,Y_s,Z_s)\overleftarrow{dB}_{s}+K_T-K_t-\int_{t}^{T}Z_sdW_{s},\;\; 0\leq t\leq T\label{1}.
\end{eqnarray}
We establish a comparison theorem for this kind of BDSDEs which help us to derive a maximal and a minimal solution when the generator $f$ is discontinuous. Our work is based to a recent paper of Bahlali et al. \cite{Bah}. They have proved that Eq. \eqref{1} has almost one solution and also a maximal and a minimal solution when the generator $f$ is continuous in $y$ and $z$.

The paper is organized as follows. In section 2, we give some notations, definitions and assumptions. Section 3 deals with our main results.

\section{ Notations, definitions and assumptions}
\setcounter{theorem}{0} \setcounter{equation}{0}
Let $(\Omega, \mathcal{F},\P)$ be a probability space in
which are defined all the processes considered in the sequel. The Euclidean norm of a vector $x \in
\mathbb{R}^{k}, (k\geq 2)$ will be denoted by $\|x\|$. \\For each $t \in
[0,T]$, we define
$\mathcal{F}_{t}\overset{\Delta}{=}\mathcal{F}_{t}^{W} \vee \mathcal{F}_{t,T}^{B},$
where for any process $\{\eta_{t}; t\in[0,T]\}$ and any $0\leq
s\leq t\leq T$; $\mathcal{F}_{s,t}^{\eta}=\sigma
\{\eta_{r}-\eta_{s}; s\leq r \leq t \} \vee \mathcal{N}$,
$\mathcal{F}_{t}^{\eta}=\mathcal{F}_{0,t}^{\eta}$; $\mathcal{N}$ denote
the class of $\P$-null sets of $\mathcal{F}$.
\\
Note that $\{\mathcal{F}_{0,t}^W, t\in [0,T]\}$ is an increasing
filtration and $\{\mathcal{F}_{t,T}^B, t\in [0,T]\}$ is a
decreasing filtration, and the collection $\{\mathcal{F}_{t}, t\in
[0,T]\}$ is neither increasing nor decreasing so it does not
constitute a filtration.

For any $n \in \mathbb{N}$, let
$\mathcal{M}^{2}(0,T,\mathbb{R}^{n})$ denote the set of ( class of
$d\P\otimes dt$ a.e. equal) $n$-dimensional jointly measurable
random processes  $\{\varphi_{t}; 0\leq t\leq T \}$  which
satisfy:
\begin{enumerate}
\item[(i)\ \ ] $\|\varphi \|_{\mathcal{M}^{2}}^{2}=\mathbb{E}(\int_{0}^{T}\mid \varphi_{t}
\mid^{2} dt)< \infty$

\item[(ii)\ \ ] $\varphi_{t}$ is $\mathcal{F}_{t}$-measurable, for
a.e. $t \in [0,T].$
\end{enumerate}
We denote by  $\mathcal{S}^{2}([0,T],\mathbb{R}^{n})$
the set of continuous $n$-dimensional random processes which
satisfy:

\begin{enumerate}
\item[(i)\ \ ] $\|\varphi \|_{\mathcal{S}^{2}}^{2}=\mathbb{E}(\underset{0\leq t\leq T}
{\sup} \mid \varphi_{t}\mid^{2})< \infty$

\item[(ii)\ \ ] $\varphi_{t}$ is $\mathcal{F}_{t}$-measurable, for any
$t \in [0,T].$
\end{enumerate}
\medskip
\begin{definition}
A solution of a RBDSDE is a triple of processes $(Y,Z, K)$ 
which satisfies Eq. \eqref{1} and such that:
\begin{enumerate}
\item[(i)\  ] $(Y,Z, K) \in
\mathcal{S}^{2}([0,T],\mathbb{R})\times
\mathcal{M}^{2}(0,T,\mathbb{R}^{ d})\times\mathrm{L}^{2}(\Omega,
\mathbb{P}, \mathbb{R}_+), $

\item[(ii)\ ] $Y_t\geq S_t$, for any
$t \in [0,T],$

\item[(iii) ] $K_t$ is continuous and increasing process with $K_0=0$ and  $\int_0^T(Y_t-S_t)dK_t=0.$
\end{enumerate}
\end{definition}
\medskip
\begin{definition}
A triple of processes $(Y_*,Z_*, K_*)$ (resp. $(Y^*,Z^*,K^*)$) of $\mathcal{S}^{2}\times
\mathcal{M}^{2}\times\mathrm{L}^{2}(\Omega)$ is said to be a minimal (resp. a maximal) solution of RBDSDE \eqref{1} if for any other solution $(Y,Z, K)$ of \eqref{1}, we have $Y_*\leq Y$ (resp.$Y\leq Y^*$).
\end{definition}
In this note, we assume that $f$ satisfies some of the following conditions:
\begin{enumerate}
\item[(H0)] $f:\Omega \times [0,T]\times \R \times
\R^{d}\rightarrow \R$ is jointly measurable satisfies $f(.,0,0)\in\mathcal{M}^{2}(0,T,\mathbb{R}^{})$ and there exists a constant $C>0$ such that for all $(t,y_i, z_i) \in [0,T]\times\mathbb{R} \times \mathbb{R}^{
 d}$, $i=1,2$ \\ $|f(t,y_1,z_1)-f(t,y_1,z_2)|\leq C(|y_1-y_2|+\|z_1-z_2\|).$
\item[(H1)] For every $(t,\omega)\in[0,T]\times \Omega$, the map $(y,z) \mapsto f(t,y,z)$  is continuous.
\item[(H2)] There exists a process $\varphi_.\in\mathcal{M}^{2}(0,T,\mathbb{R})$ with positive values and a positive constant $\kappa>0$ such that\ $|f(t,y,z)|\leq \varphi_t+\kappa(|y|+\|z\|)$, for all $(t,y, z) \in[0,T]\times\mathbb{R} \times \mathbb{R}^{d}$.
\item[(H3)] For every $(t,\omega)\in[0,T]\times \Omega$, $z \in \mathbb{R}^{
 d}$, the map $y \mapsto f(t,y,z)$ is left-continuous and non-decreasing and for $y \in \mathbb{R}$, $z \mapsto f(t,y,z)$ is continuous.
\item[(H4)] There exists a continuous function
$h:\mathbb{R} \times \mathbb{R}^{d}\rightarrow \mathbb{R}$, which
satisfies

$|h(y,z)|\leq \kappa(|y|+\|z\|)$\  for any  $(y, z) \in \mathbb{R}
\times \mathbb{R}^{
 d}$, such that for all $y_1\geq y_2$, $t \in[0,T]$ , $z_1, z_2 \in \mathbb{R}^{
 d}$, we have
$ f(t,y_{1},z_{1})-f(t,y_{2},z_{2}) \geq
h(y_{1}-y_{2},z_{1}-z_{2})$
 \end{enumerate}
Moreover, we assume that:
\begin{enumerate}
\item[(H5)] The terminal condition $\xi$ belongs to $\mathrm{L}^{2}(\Omega,
\mathcal{F}_{T}, \mathbb{P}, \mathbb{R})$.
\item[(H6)] The obstacle $S$ belongs to $ \in\mathcal{S}^{2}([0,T],\mathbb{R}^{})$ such that $S_T\leq \xi$ a.s.
\item[(H7)] $g:\Omega \times [0,T]\times \mathbb{R} \times
\mathbb{R}^{d}\rightarrow \mathbb{R}^l$ is jointly measurable satisfies $g(.,0,0)\equiv0$ and there exist two constants $C
> 0$ and  $0 < \alpha < 1$ such that for all $t \in [0,T] , (y_{1},z_{1}), (y_{2}, z_{2}) \in
\mathbb{R} \times \mathbb{R}^{ d}$,\\
$\|g(t,y_{1},z_{1})-g(t,y_{2},z_{2})\|^{2} \leq
C\mid y_{1}-y_{2}\mid^{2}+\alpha \| z_{1}-z_{2}\|^{2}.$
 \end{enumerate}

\section{Main results}
Our purpose is to establish an existence of minimal or maximal solution to RBDSDEs \eqref{1} when parameters $(f, g, \xi, S)$
satisfy (H2)-(H7).

To attain our goal, we need to establish first the following theorem which is an extension of the existence
result established in Bahlali et al. \cite{Bah}.
\begin{theorem}\label{l0}
Assume that $(H1)$-$(H2)$ and $(H5)$-$(H7)$ hold. Then, the RBDSDE \eqref{1} has a solution. Moreover, there is a minimal and a maximal solution to RBDSDE \eqref{1}.
\end{theorem}
\begin{proof}
We define $f_{n}(t,y,z)=\underset{u \in \Q}\inf \Big\{f(t,u,z)+
n\mid y-u\mid \Big\}$, for $n\geq \kappa$. For every $n\geq \kappa$, $f_{n}$ is uniformly $n$-Lipschitz and ($f_{n}$) converges suitably to $f$. Now, as $|f(t,y,z)|\leq \varphi_t+\kappa(|y|+\|z\|)$, the rest of is the adaptation of Theorem 3.3 in Bahlali et al. \cite{Bah} where $|f(t,y,z)|\leq \kappa(1+|y|+\|z\|)$. Therefore it is ommitted.
\end{proof}
We also need the following comparison results.
\begin{theorem}[Comparison with at least one Lipschitz function]\label{l0a}
Let $g$, $S^i$ and $\xi^i$ (i=1,2) satisfy (H5)-(H7). Assume that RBDSDEs $(f^1, g, \xi^1, S^1)$ and $(f^2, g, \xi^2, S^2)$ have solutions $(Y^1,Z^1, K^1)$ and $(Y^2,Z^2, K^2)$, respectively. Assume moreover that:
\medskip
\begin{enumerate}
\item[(i)\  ] $\xi^1\leq\xi^2$\ \ a.s.,

\item[(ii)\ ] $S_t^1\leq S_t^2$\ \ a.s., for all
$t \in [0,T]$

\item[(iii) ] $f^1$ satisfies (H0) such that $f^1(t,Y^2,Z^2)\leq f^2(t,Y^2,Z^2)$\ \ a.s.

 (resp. $f^2$ satisfies (H0) such that $f^1(t,Y^1,Z^1)\leq f^2(t,Y^1,Z^1)$\ \ a.s.).
 \end{enumerate}
 \medskip
 Then, $Y_t^1\leq Y_t^2$\ \ a.s., for all $t\in[0,T]$.
\end{theorem}
\begin{proof}
Applying Itô's formula to
$|(Y_{t}^{1}-Y_{t}^{2})^{+}|^2$, we have
\begin{eqnarray*}
&&\E|(Y_{t}^{1}- Y_{t}^{2})^{+}|^{2}+\E\int_{t}^{T}\textbf{1}_{\{Y_{s}^{1}>Y_{s}^{2}\}}\|Z_{s}^{1}-Z_{s}^{2}\|^{2}ds\\
&=&\E|(\xi^{1}- \xi^{2})^{+}|^{2}+2\E\int_{t}^{T}( Y_{s}^{1}-Y_{s}^{2})^{+}\left(f^{1}(s,Y_{s}^{1},Z_{s}^{1})-f^{2}(s,Y_{s}^{2},Z_{s}^{2})\right)ds\\
&&+2\E\int_{t}^{T}( Y_{s}^{1}-Y_{s}^{2})^{+}(dK_{s}^{1}-dK_{s}^{2})
+\E\int_{t}^{T}\textbf{1}_{\{Y_{s}^{1}>Y_{s}^{2}\}}\|g(s,Y_{s}^{1},Z_{s}^{1})-g(s,Y_{s}^{2},Z_{s}^{2})\|^{2}ds.\nonumber
\end{eqnarray*}
From (i), $\E|(\xi^{1}- \xi^{2})^{+}|^{2}=0$ and from (iii), we have
\begin{eqnarray*}
f^{1}(s,Y_{s}^{1},Z_{s}^{1})-f^{2}(s,Y_{s}^{2},Z_{s}^{2})\leq f^{1}(s,Y_{s}^{1},Z_{s}^{1})-f^{1}(s,Y_{s}^{2},Z_{s}^{2}).
\end{eqnarray*}
Therefore, from Young inequality, and the fact that $f^1$ satisfies $(H0)$ and $g$ verify $(H7)$, we get
\begin{eqnarray*}
&&\E|(Y_{t}^{1}- Y_{t}^{2})^{+}|^{2}+\E\int_{t}^{T}\textbf{1}_{\{Y_{s}^{1}>Y_{s}^{2}\}}\|Z_{s}^{1}-Z_{s}^{2}\|^{2}ds \\&\leq&\left(\frac{1}{\beta}+\beta C+C\right)\E\int_{t}^{T}|( Y_{s}^{1}-Y_{s}^{2})^{+}|^2ds+(\beta C+\alpha)\E\int_{t}^{T}\textbf{1}_{\{Y_{s}^{1}>Y_{s}^{2}\}}\|Z_{s}^{1}-Z_{s}^{2}\|^2ds
\\&&+2\E\int_{t}^{T}( Y_{s}^{1}-Y_{s}^{2})^{+}(dK_{s}^{1}-dK_{s}^{2}).\nonumber
\end{eqnarray*}
Since \ $Y_{t}^{1}>S_{t}^{2}\geq S_{t}^{1}$ \ on the set $\{Y_{s}^{1}>Y_{s}^{2}\}$ 
we derive that
\begin{eqnarray*}
\E\int_{t}^{T}( Y_{s}^{1}-Y_{s}^{2})^{+}(dK_{s}^{1}-dK_{s}^{2})=-\E\int_{t}^{T}( Y_{s}^{1}-Y_{s}^{2})^{+}dK_{s}^{2}\leq0.
\end{eqnarray*}
Hence,
\begin{eqnarray*}
&&\E|(Y_{t}^{1}- Y_{t}^{2})^{+}|^{2}+\E\int_{t}^{T}\textbf{1}_{\{Y_{s}^{1}>Y_{s}^{2}\}}\|Z_{s}^{1}-Z_{s}^{2}\|^{2}ds \\&\leq&\left(\frac{1}{\beta}+\beta C+C\right)\E\int_{t}^{T}|( Y_{s}^{1}-Y_{s}^{2})^{+}|^2ds+(\beta C+\alpha)\E\int_{t}^{T}\textbf{1}_{\{Y_{s}^{1}>Y_{s}^{2}\}}\|Z_{s}^{1}-Z_{s}^{2}\|^2ds.\nonumber
\end{eqnarray*}
Consequently, choosing $0<\beta<\frac{1-\alpha}{C}$ and using Gronwall inequality, we obtain
\medskip
$
\E|(Y_{t}^{1}- Y_{t}^{2})^{+}|^{2}\leq0. \ \ \text{Thus}\ \ \
(Y_{t}^{1}- Y_{t}^{2})^{+}=0\ \ a.s.\ \ i.e.\ \ Y_{t}^{1}\leq Y_{t}^{2}\ \ a.s.,\ \ \forall\ t\in[0,T].\ \ \
$
\end{proof}
\begin{theorem}[Comparison with at least one continuous function]\label{o}
Let $g$, $S^i$ and $\xi^i$ (i=1,2) satisfy (H5)-(H7). Assume that RBDSDEs $(f^1, g, \xi^1, S^1)$ and $(f^2, g, \xi^2, S^2)$ have solutions $(Y^1,Z^1, K^1)$ and $(Y^2,Z^2, K^2)$, respectively. Assume moreover that:
\begin{enumerate}
\item[(i)\  ] $\xi^1\leq\xi^2$\ \ a.s.,

\item[(ii)\ ] $S_t^1\leq S_t^2$\ \ a.s., for all
$t \in [0,T],$

\item[(iii) ] $f^1$ satisfies (H1)-(H2) such that $f^1(t,Y^2,Z^2)\leq f^2(t,Y^2,Z^2)$ a.s. and $(Y^1,Z^1, K^1)$ is the minimal solution\ \
 (resp. $f^2$ satisfies (H1)-(H2) such that $f^1(t,Y^1,Z^1)\leq f^2(t,Y^1,Z^1)$ a.s. and $(Y^2,Z^2, K^2)$ is the maximal solution).
 \end{enumerate}
Then, $Y_t^1\leq Y_t^2$\ \ a.s., for all $t\in[0,T]$.
\end{theorem}
\begin{proof} For any fixed $\kappa>0$, let us define
$$f_{n}^1(t,y,z)=\underset{u \in \Q}\inf \Big\{f^1(t,u,z)+
n\mid y-u\mid \Big\},\;\; \forall\; n\geq \kappa.$$ Hence, for every $n\geq \kappa$, $f_{n}^{1}$ is uniformly $n$-Lipschitz, linear growth and converges
suitably to $f^1$ (cf. Lepeltier and San Martin \cite{LSm}). Then we get from Theorem \ref{l0} that for every $n\geq \kappa$, RBDSDE $(f_n^1, g, \xi^1, S^1)$
has a unique adapted solution $(Y^{1,n}, Z^{1,n}, K^{1,n} )$ which converges to the minimal solution $(Y^1,Z^1, K^1)$ to the  RBDSDE $(f^1, g, \xi^1, S^1)$
(cf. proof of Theorem 3.3 in Bahlali et al. \cite{Bah}). Moreover, for all $n\geq \kappa$, $f_n^1\leq f^1$. Therefore, from $(iii)$, we have  $f_n^1(t,Y^2,Z^2)\leq f^2(t,Y^2,Z^2)$\ a.s. Then, by Theorem \ref{l0a}, we get  $Y^{1,n}\leq Y^2$ a.s., for all $n\geq \kappa$. Hence, we have $Y^{1}\leq Y^2$.

On the other hand, if we define $$f_{n}^2(t,y,z)=\underset{u \in \Q}\sup \Big\{f^2(t,u,z)-
n\mid y-u\mid \Big\}\;\;\forall\ n\geq \kappa,$$
it is easy to check that for all $n\geq \kappa$, $f^1(t,Y^1,Z^1)\leq f^2(t,Y^1,Z^1)\leq f_n^2(t,Y^1,Z^1)$ and $f_{n}^2$ is uniformly $n$-Lipschitz, linear growth and converges suitably to $f^2$. Then, applying again Theorem \ref{l0a}, $Y^{1}\leq Y^{2,n}$ a.s., for all $n\geq \kappa$, where $(Y^{2,n}, Z^{2,n}, K^{2,n})$ is the unique solution to BDSDEs $(f_n^2, g, \xi^2,S^2)$ which converges to $(Y^{2}, Z^{2}, K^{2})$, the maximal solution of BDSDEs $(f^2, g, T, \xi^2)$ (cf. proof of Theorem 3.3 in Bahlali et al. \cite{Bah}). Therefore, we get $Y^{1} \leq Y^{2}$ a.s.
\end{proof}
\begin{lemma}\label{l1}
Let $\phi$ belongs in $\mathcal{M}^{2}(0,T; \mathbb{R})$ and $h$ appear in assumption $(H5)$. For a continuous function of finite variation $A$ belongs in $\mathrm{L}^{2}(\Omega,
\mathbb{P}, \mathbb{R})$ and verifies $A_0=0$, we consider  the processes $(\bar{Y},\bar{Z})\in\mathcal{S}^{2}([0,T],\mathbb{R})\times
\mathcal{M}^{2}(0,T,\mathbb{R}^{ d})$ such that:
\begin{eqnarray}\label{2}
&&(i)\;\bar{Y}_t=\xi+\int_{t}^{T}\left[h(\bar{Y}_s,\bar{Z}_s)+\phi_s\right]ds+\int_{t}^{T}g(s,\bar{Y}_s,\bar{Z}_s)\overleftarrow{dB}_{s}+A_T-A_t-\int_{t}^{T}\bar{Z}_sdW_{s},
,\ \ t \in [0,T]\nonumber\\
&&(ii)\;\int_0^T\bar{Y}_{s}^{-}dA_{s}\geq0.
\end{eqnarray}
Then, if $\phi_t\geq0$ and $\xi \geq 0$, we have $\bar{Y}_t\geq 0,$ \    $\P$-a.s. \ $\forall\ t\in [0,T]$.
\end{lemma}
\begin{remark}
Let us note that the assumption $(ii)$ in Lemma \ref{l1} is not a technic hypothesis but becomes natural since we are in our framework i.e $\int^T_0(Y_s-S_s)dK_s=0$, where $(Y,Z,K)$ is a solution of
\end{remark}
\begin{proof}
Applying Itô's formula to
$|Y_{t}^{-}|^2$, we have
\begin{eqnarray*}
\E|\bar{Y}_{t}^{-}|^{2}+\E\int_{t}^{T}\textbf{1}_{\{\bar{Y}_{s}< 0\}}\|\bar{Z}_{s}\|^{2}ds &=&\E| \xi^{-}|^{2}-2\E\int_{t}^{T}\bar{Y}_{s}^{-}\Big(h(\bar{Y}_s,\bar{Z}_s)+\phi_s\Big)ds
\\&&-2\E\int_{t}^{T}\bar{Y}_{s}^{-}dA_{s}
+\E\int_{t}^{T}\textbf{1}_{\{\bar{Y}_{s}< 0\}}\|g(s,\bar{Y}_{s},\bar{Z}_{s})\|^{2}ds.\nonumber
\end{eqnarray*}
Since $\phi_t\geq0$ and $\xi \geq 0$ and using the fact that $-2\E\int_{t}^{T}\bar{Y}_{s}^{-}dA_{s}\leq0$, we derive that
\begin{eqnarray*}
\E|\bar{Y}_{t}^{-}|^{2}+\E\int_{t}^{T}\textbf{1}_{\{\bar{Y}_{s}< 0\}}\|\bar{Z}_{s}\|^{2}ds &\leq&-2\E\int_{t}^{T}\bar{Y}_{s}^{-}h(\bar{Y}_s,\bar{Z}_s)ds
+\E\int_{t}^{T}\textbf{1}_{\{\bar{Y}_{s}< 0\}}\|g(s,\bar{Y}_{s},\bar{Z}_{s})\|^{2}ds.\nonumber
\end{eqnarray*}
From $(H7)$, we get $\|g(s,y,z)\|^2\leq C|y|^2+\alpha\|z\|^2$ which together with $(H4)$ and Young inequality provide
\begin{eqnarray*}
\E|\bar{Y}_{t}^{-}|^{2}+\E\int_{t}^{T}\textbf{1}_{\{\bar{Y}_{s}< 0\}}\|\bar{Z}_{s}\|^{2}ds &\leq&\left(\frac{1}{\beta}+2\beta \kappa^2+C\right)\E\int_{t}^{T}|\bar{Y}_{s}^{-}|^2ds+(2\beta \kappa^2+\alpha)\E\int_{t}^{T}\textbf{1}_{\{\bar{Y}_{s}< 0\}}\|\bar{Z}_{s}\|^2ds
.\nonumber\displaystyle
\end{eqnarray*}
Therefore, choosing $0<\beta<\frac{1-\alpha}{2\kappa^{2}}$ and using Gronwall inequality, we obtain
$\bar{Y}_{t}^{-}=0\ \ \P$-a.s.\ \ $\forall\ t\in[0,T]$, which implies that $\ \ \bar{Y}_{t}\geq 0\ \ \P$-a.s. $\forall\ t\in[0,T]$.
\end{proof}
Now, we are ready to prove our main result.
\begin{theorem}\label{owo}
Under assumptions $(H2)$-$(H7)$, the RBDSDE \eqref{1} has at least one solution. Also, there is a minimal solution $(\underline{y},\underline{z}, \underline{k})$ to RBDSDE \eqref{1}.
\end{theorem}
\begin{proof}
By virtue of Theorem \ref{l0}, let consider the processes $(y^0,z^0, k^{0})$, $(\tilde{y}^0,\tilde{z}^0, \tilde{k}^0)$ and the sequence of processes $\left\{(y^n,z^n, k^{n})\right\}_{n\geq1}$ respectively minimal solution of the following RBDSDE: for all $t\in[0,T]$,
\begin{eqnarray}\label{eq3}
\left\{
  \begin{array}{ll}
    & y_{t}^{0}=\xi
+\int_{t}^{T}(-\kappa|y_{s}^{0}|-\kappa\|z_{s}^{0}\|-\varphi_s)ds+
k_{T}^{0}-k_{t}^{0}
+\int_{t}^{T}g(s,y_s^{0},z_s^{0})\overleftarrow{dB}_{s}
-\int_{t}^{T}z_s^{0}dW_{s},\\
     &y_{t}^{0}\geq S_{t},\\
    &\int_0^T(y_{s}^{0}-S_{s})dk_{s}^{0}=0
  \end{array}
\right.
\end{eqnarray}

\begin{eqnarray}\label{eq5}
\left\{
  \begin{array}{ll}
    & \tilde{y}_{t}^{0}=\xi
+\int_{t}^{T}(\kappa|\tilde{y}_{s}^{0}|+\kappa\|\tilde{z}_{s}^{0}\|+\varphi_s)ds+
\tilde{k}_{T}^{0}-\tilde{k}_{t}^{0}
+\int_{t}^{T}g(s,\tilde{y}_s^{0},\tilde{z}_s^{0})\overleftarrow{dB}_{s}
-\int_{t}^{T}\tilde{z}_s^{0}dW_{s}, \ t \in [0,T]\hbox{,} \\
     &\tilde{y}_{t}^{0}\geq S_{t},\ \ \ t \in [0,T] \hbox{,} \\
    &\int_0^T(\tilde{y}_{s}^{0}-S_{s})d\tilde{k}_{s}^{0}=0  \hbox{.}
  \end{array}
\right.
\end{eqnarray}
and
\begin{eqnarray}\label{11}
\left\{
  \begin{array}{ll}
    & y_t^n=\xi+\int_{t}^{T}\left(f(s,y_s^{n-1},z_s^{n-1})+
h(y_s^n-y_s^{n-1},z_s^n-z_s^{n-1})\right)ds+k_{T}^{n}-k_{t}^{n}
+\int_{t}^{T}g(s,y_s^{n},z_s^n)\overleftarrow{dB}_{s}-\int_{t}^{T}z_s^ndW_{s},\\
     &y_{t}^n\geq S_{t},\\
    &\int_0^T(y_{s}^n-S_{s})dk_{s}^n=0  \hbox{.}
  \end{array}
\right.
\end{eqnarray}
To complete the proof, it's suffice to show that the sequence $(y^{n},z^{n}, k^{n})$ converges to a limit  $(\underline{y},\underline{z}, \underline{k})$ which is the minimal solution of RBDSDE \eqref{1}. In this end, we shall first prove that for any\ $n\geq 0$,
\begin{eqnarray*}
y_t^{n}\leq y_t^{n+1}
\leq \tilde{y}_t^{0}
    ,\ \ \ \P\mbox{-a.s}. \ \  \forall\ t\in [0,T].
\end{eqnarray*}
For $n\geq0$, we set $(y_t^{n+1,n},z_t^{n+1,n},k_t^{n+1,n})=(y_t^{n+1}-y_t^{n},z_t^{n+1}-z_t^{n},k_t^{n+1}-k_t^{n})$, which satisfies the following equation:
\begin{eqnarray*}\label{}
y_t^{n+1,n}&=&\int_{t}^{T}\left(h(y_s^{n+1,n},z_s^{n+1,n})+
\phi_s^n\right)ds\notag+k_T^{n+1,n}-k_t^{n+1,n}+\int_{t}^{T}g^{n}(s,y_s^{n+1,n},z_s^{n+1,n})\overleftarrow{dB}_{s}\\&&-
\int_{t}^{T}z_s^{n+1,n}dW_{s},
\end{eqnarray*}
where $\displaystyle{g^n(t,y,z)=g(t,y+y_{t}^n,z+z_{t}^n)-g(t,y_{t}^n,z_{t}^n),\,\forall\,n\geq0,\;  \phi_s^0=f(s,y_{s}^{0},z_{s}^{0})+\kappa|y_{s}^{0}|+\kappa\|z_{s}^{0}\|+\varphi_s}$ and $\displaystyle{\phi_s^n=f(s,y_{s}^{n},z_{s}^{n})-f(s,y_{s}^{n-1},z_{s}^{n-1})-h(y_s^n-y_s^{n-1},z_s^n-z_s^{n-1}),\;\; n\geq1}$. According to it definition, one can show that $\phi^0$ and $g^n, \;\forall\ n\geq 0$ satisfy all assumptions of Lemma \ref{l1}.
Moreover, since $k_t^{n}$ is a continuous and increasing process, for all $n\geq 0,\; k^{n+1,n}$ is a continuous processes of finite variation and, using the same argument as one appear in \cite{Bah}, on can show that
\begin{eqnarray*}
\int_{0}^{T}(y_t^{n+1,n})^{-}dk_t^{n+1,n}=
\int_{0}^{T}(y_t^{n+1}-
y_t^{n})^{-}dk_t^{n+1}\geq0.
\end{eqnarray*}
Therefore, it follows from Lemma \ref{l1} that $y_t^{1,0}\geq 0$ a.s., i.e. $y_t^{0}\leq y_t^{1}$, a.s.for all $t\in [0,T]$.
Let us suppose that there exists $n\geq 1$ such that  $y_t^{n-1}\leq y_t^{n}$. Then, for such $n,\,\phi^n$ satisfies assumption of Lemma \ref{l1} from which, we obtain $y_t^{n}\leq y_t^{n+1}$, a.s., for all $t\in [0,T]$. Finally, for all $n\geq 0$, $y_t^{n}\leq y_t^{n+1}$ a.s. for all $t\in [0,T]$.

Setting $(\tilde{y}_t^{0,n},\tilde{z}_t^{0,n},\tilde{k}_t^{0,n})=(\tilde{y}_t^{0}-y_t^{n},\tilde{z}_t^{0}-z_t^{n},\tilde{k}_t^{0}-k_t^{n})$, we check similarly as above that for all $n\geq 0$, $\tilde{y}_t^{0,n}\geq 0$  a.s., for all $t\in[0,T]$, i.e. for all $n\geq 0, \; y_t^{n}\leq \tilde{y}^0_t$,  a.s., for all $t\in[0,T]$.
Thus, we have for all $n\geq 0$, $$y_t^{n}\leq y_t^{n+1}\leq \tilde{y}_t^{0},\ \ \ \P\mbox{-a.s.} \ \  \forall\ t\in [0,T].
$$
Moreover, since $|y_t^{n}|\leq \max(|\tilde{y}_t^{0}|,|y_t^{0}|),\; \forall\, t\in[0,T]$ we have
\begin{eqnarray}\label{A0} \underset{n}{\sup}\
\E\Big(\underset{0\leq t\leq
T}{\sup}|y_t^n|^2\Big)\leq\max\left(\E\Big(\underset{0\leq
t\leq T}{\sup}|\tilde{y}_t^0|^2\Big),\E\Big(\underset{0\leq
t\leq T}{\sup}|y_t^0|^2\Big)\right)<+\infty.
\end{eqnarray}
Therefore, we deduce from the dominated convergence theorem that
$(y_s^{n})_{n\geq0}$ converges in $\mathcal{S}^{2}([0,T],\mathbb{R})$ to a limit $\underline{y}$.

On the other hand, by virtue of Itô's formula, we have
\begin{eqnarray}
\E\left(|y_0^{n+1}|^2+\int_{0}^{T}\|z_s^{n+1}\|^2ds\right)&=&\E|\xi|^2
+2\E\int_{0}^{T}y_s^{n+1}\left(f(s,y_s^{n},z_s^{n})+
h(y_s^{n+1}-y_s^{n},z_s^{n+1}-z_s^{n})\right)ds
\notag\\&&+2\E\int_{0}^{T}y_s^{n+1}dk_s^{n+1}+\E\int_{0
}^{T}\|g(s,y_s^{n+1},z_s^{n+1})\|^2ds.\label{A}
\end{eqnarray}
From (H2), (H4), (H7) and Young inequalities, we get for any $\gamma, \sigma>0$,
\begin{eqnarray*}\label{}
y_s^{n+1}\left(f(s,y_s^{n-1},z_s^{n-1})+h(y_s^{n}-y_s^{n-1},z_s^{n}-z_s^{n-1})\right)&\leq&|y_s^{n+1}|\varphi_s+
\kappa|y_s^{n+1}|\left(
2|y_s^{n}|+2\|z_s^{n}\|+|y_s^{n+1}|
+\|z_s^{n+1}\|\right)\\
&\leq&\left(\frac{1}{2}+\kappa^2+
\frac{2\kappa^2}{\gamma}+\kappa+\frac{\kappa^2}{2\sigma}\right)
|y_s^{n+1}|^2+|y_s^{n}|^2+\frac{\gamma}{2}\|z_s^{n}\|^2\\
&&+\frac{\sigma}{2}\|z_s^{n+1}\|^2+\frac{1}{2}|\varphi_s|^2,\\
\|g(s,y_s^{n+1},z_s^{n+1})\|^2&\leq&
C|y_s^{n+1}|^2 +\alpha\|z_s^{n+1}\|^2.
\end{eqnarray*}
Using again Young inequality, we have for any $\beta>0$,
\begin{eqnarray*}
2\E\int_{0}^{T}y_s^{n+1}dk_s^{n+1}
=2\int_{0}^{T}S_sdk_s^{n+1}
\leq\frac{1}{\beta}\E\left(\underset{0\leq t\leq
T}{\sup}|S_s|^2\right)+\beta\E\left(k_T^{n+1}\right)^2
\end{eqnarray*}
Therefore, there exists a constant $C_1$ independent of $n$ such that
for any \ $\gamma, \sigma>0$, we derive
\begin{equation}\label{B}
\E\int_{0}^{T}\|z_s^{n+1}\|^2ds\leq C_1
+(\sigma+\alpha)\E\int_{0}^{T}\|z_s^{n+1}\|^2ds
+\gamma\E\int_{0 }^{T}\|z_s^{n}\|^2ds+\beta\E|k_T^{n+1}|^2.
\end{equation}

Moreover, since
\begin{eqnarray*}\notag
k_T^{n+1}&=&y_0^{n+1}-\xi-
\int_{0}^{T}\left(f(s,y_s^{n},z_s^{n})+
h(y_s^{n+1}-y_s^{n},z_s^{n+1}-z_s^{n})\right)ds
-\int_{0}^{T}g(s,y_s^{n+1},z_s^{n+1})\overleftarrow{dB}_{s}
\\&&\hspace{1cm}+\int_{0}^{T}z_s^{n+1}dW_{s},\ \ t \in [0,T],
\end{eqnarray*}
it follows from Hölder and BDG's inequalities and the properties on $f$, $h$ and $g$ that there exists $C_2$ independent of $n$ such that
\begin{eqnarray}\label{d11}
\E\left(k_T^{n+1}\right)^2&\leq&C_2+c\E\int_{0}^{T}\left(\|z_s^{n}\|^2
+\|z_s^{n+1}\|^2\right)ds\hbox{.}
\end{eqnarray}
According to \eqref{B} and \eqref{d11} and choosing  $\sigma>0$ and $\beta>0$ such that
$0<\sigma+\beta c<1-\alpha$, we derive for any $\gamma>0$, $n\geq 0$
\begin{eqnarray*}\label{R}
\E\int_{0}^{T}\|z_s^{n+1}\|^2ds\leq
\frac{\Lambda}{1-\alpha-\sigma-\beta c} +\frac{\gamma+\beta c}{1-\alpha-\sigma-\beta c}\E\int_{0
}^{T}|z_s^{n}|^2ds,
\end{eqnarray*}
which provide by iteration
\begin{eqnarray}
\E\int_{0}^{T}\|z_s^{n+1}\|^2ds&\leq&
\frac{\Lambda}{1-\alpha-\sigma-\beta c}\sum_{i=0}^{n-1}\left(\frac{\gamma+\beta c}{1-\alpha-\sigma-\beta c}\right)^i\notag\\
&&+\left(\frac{\gamma+\beta c}{1-\alpha-\sigma-\beta c}\right)^n\E\int_{0 }^{T}\|z_s^{0}\|^2ds.\label{R}
\end{eqnarray}
Choosing $\gamma>0$ (for example, one can take: $0<\sigma<\displaystyle\frac{1-\alpha}{2}$, $0<\beta<\displaystyle\frac{1-\alpha}{4c}$ and $0<\gamma<1-\alpha-\sigma-2\beta c$ ) such that
$0<\displaystyle\frac{\gamma+\beta c}{1-\alpha-\sigma-\beta c}<1$ and noting that
\ $\displaystyle\E\int_{0 }^{T}\|z_s^{0}\|^2ds<\infty$, we obtain
\begin{eqnarray}
\underset{n\geq0}{\sup}\
\E\int_{0}^{T}\|z_s^{n}\|^2ds<+\infty.\label{C}
\end{eqnarray}
Denoting $\theta^n_s=f(s,y_s^{n-1},z_s^{n-1})+h(y_s^{n}-y_s^{n-1},z_s^{n}-z_s^{n-1})$, it follows from \eqref{A0} and \eqref{C} that
$\theta^n_s$ is uniformly bounded in $\mathcal{M}^{2}(0,T,\mathbb{R})$.

Applying again Itô's formula to $|y_{t}^{p}-y_{t}^{n}|^2$, we have
\begin{eqnarray*}
\E|y_{t}^{p}- y_{t}^n |^{2}+\E\int_{t}^{T}\|z_{s}^{p}-z_{s}^n\|^{2}ds &
=&2\E\int_{t}^{T}( y_{s}^{p}-y_{s}^n )(\theta_s^p-\theta_s^n)ds+2\E\int_{t}^{T}( y_{s}^{p}-y_{s}^n )(dk_{s}^{p}-dk_{s}^{n})\\&&
+\E\int_{t}^{T}\|g(s,y_{s}^{p},z_{s}^{p}-)
g(s,y_{s}^{n},z_{s}^{n})\|^{2}ds.\nonumber
\end{eqnarray*}
Using the fact that $y_{t}^n\geq S_{t}$ for all\ $ t \in [0,T]$ and the identity $\displaystyle \int_0^T(y_{s}^n-S_{s})dk_{s}^n=0$, we obtain
\begin{eqnarray*}
\E\int_{0}^{T}\|z_{s}^{p}-z_{s}^n\|^{2}ds &
\leq&2\E\int_{0}^{T}( y_{s}^{p}-y_{s}^n )(\theta_s^p-\theta_s^n)ds+\E\int_{0}^{T}\|g(s,y_{s}^{p},z_{s}^{p})
-g(s,y_{s}^{n},z_{s}^{n})\|^{2}ds.\nonumber
\end{eqnarray*}
Therefore, by virtue of Hölder's inequality and (H7), we obtain
\begin{eqnarray*}
(1-\alpha)\E\int_{0}^{T}\|z_{s}^{p}-z_{s}^n\|^{2}ds &\leq& 4\left(\underset{n\geq0}{\sup}\|\theta^n\|_{\mathcal{M}^{2}}\right)\left(\E\int_{0}^{T}| y_{s}^{p}-y_{s}^n |^2ds\right)^{\frac{1}{2}}
+C\E\int_{0}^{T}|y_{s}^{p}-y_{s}^{n}|^{2}ds,\nonumber
\end{eqnarray*}
which yields that $\left(z^{n}\right)_{n\geq0}$ is a Cauchy
sequence in $\mathcal{M}^{2}(0,T,\mathbb{R}^{ d})$ so that it  converges in
$\mathcal{M}^{2}(0,T,\mathbb{R}^{ d})$ to a limit $\underline{z}$. On the other hand, since $(y^n,z^n)\rightarrow (\underline{y},\underline{z})$ in $\mathcal{M}^2(\R^d)\times\mathcal{S}(\R)$, then there exists $(y',z')\in M^2(\R^d)\times\mathcal{S}(\R)$ and a subsequence
which we still denote $(y^n,z^n)$ such that $\forall n, |y^n|<y',\, \|z^n\|<z'$ and $(y^n,z^n) \rightarrow (\underline{y},\underline{z}), dt\times d\P$ a.e.
Therefore, from the properties of $f$, $g$ and $h$, we get for almost all $\omega$,
\begin{eqnarray*}f(t,y_t^{n-1},z_t^{n-1})+
h(y_t^{n}-y_t^{n-1},z_t^{n}-z_t^{n-1})
\longrightarrow f(t,\underline{y}_{t},\underline{z}_{t}),
\end{eqnarray*}
$\P$-a.s., for all $t\in[0,T]$ as $n\rightarrow \infty$.
Then, it follows by the dominated convergence theorem that
\begin{eqnarray*}\E\int_{t}^{T}|f(s,y_s^{n-1},z_s^{n-1})+
h(y_s^{n}-y_s^{n-1},z_s^{n}-z_s^{n-1})
- f(s,\underline{y}_{s},\underline{z}_{s})|^2ds\rightarrow 0
\end{eqnarray*}
as $n \rightarrow\infty$.
On the other hand, by Burkhölder-Davis Gundy inequality,
\begin{eqnarray*}
&&\E\underset{0\leq t\leq T}\sup\left|\int_{t}^{T}g(s,y_s^{n},z_s^{n})dB_s- \int_{t}^{T}g(s,\underline{y}_{s},\underline{z}_{s})dB_s \right|^2\\&&\leq C\E\int_{0}^{T}\left|y_{s}^{n}- \underline{y}_{s}\right|^2ds+\alpha\E\int_{0}^{T}\|z_{s}^{n}- \underline{z}_{s}\|^2ds\underset{n \longrightarrow
\infty}{\longrightarrow} 0,\end{eqnarray*}
and
\begin{eqnarray*}\E\underset{0\leq t\leq T}\sup\left|\int_{t}^{T}z_{s}^{n}dW_s- \int_{t}^{T}\underline{z}_{s}dW_s \right|^2\leq\E\int_{0}^{T}\left|z_{s}^{n}- \underline{z}_{s}\right|^2ds\underset{n \longrightarrow
\infty}{\longrightarrow} 0.
\end{eqnarray*}
Since, $(y^n,z^n, \theta^n)$ converges in $\mathcal{S}^{2}([0,T];\mathbb{R})\times
\mathcal{M}^{2}([0,T];\mathbb{R}^{ d})\times
\mathcal{M}^{2}([0,T];\mathbb{R}^{ d})$
and
\begin{eqnarray*}
\E\underset{0\leq t\leq T}\sup\left|k_{t}^p-k_{t}^n\right|^2
&\leq&\E|y_{0}^p-y_{0}^n|^2+\E\underset{0\leq t\leq T}\sup\left|y_{t}^p-y_{t}^n\right|^2+
\E\displaystyle\int_{0}^{T}|\theta_s^p-\theta_s^n|^2ds\\
&&+\E\underset{0\leq t\leq T}\sup\left|\int_{0}^{t}\left(g(s,y_{s}^p,z_{s}^p)
-g(s,y_{s}^n,z_{s}^n)\right)
\overleftarrow{dB}_{s}\right|^2+\E\underset{0\leq t\leq T}\sup\left|\int_{0}^{t}(z_{s}^p-z_{s}^n)dW_{s}\right|^2,
\end{eqnarray*}
for any $n,p\geq0$, we deduce from Burkhölder-Davis Gundy inequality that
\begin{eqnarray*}
\E\left(\underset{0\leq t\leq T}\sup\left|k_{t}^p-k_{t}^n\right|^2\right)\rightarrow 0,
\end{eqnarray*}
as $n,\, p \rightarrow\infty$. Consequently, there exists a $\mathcal{F}_t$-mesurable process $k$ with value in $\R$ such that
\begin{eqnarray*}
\E\left(\underset{0\leq t\leq T}\sup\left|k_{t}^n-\underline{k}_{t}\right|^2\right)\longrightarrow 0,
\end{eqnarray*}
as $n \rightarrow\infty$. Obviously, $\underline{k}_{0}=0$ and $\{\underline{k}_{t};\ 0\leq t \leq T\}$ is a non-decreasing and continuous process.
From \eqref{11}, we have for all $n\geq0$,  $y_{t}^n\geq S_{t},\ \ \forall\ t \in [0,T] $, then $\underline{y}_{t}\geq S_{t},\ \ \forall\ t \in [0,T].$

On the other hand, from the result of Saisho \cite{Saisho} (1987, p. 465), we have
\begin{eqnarray*}
\int_0^T(y_{s}^n-S_{s})dk_{s}^n\ \rightarrow\int_0^T(\underline{y}_{s}-S_{s})dk_{s}
\end{eqnarray*}
$\P$-a.s. as $n\rightarrow\infty$. Using the identity $\displaystyle\int_0^T(y_{s}^n-S_{s})dk_{s}^n=0$, for all $n\geq 0$, we obtain\newline $\displaystyle\int_0^T(\underline{y}_{s}-S_{s})d\underline{k}_{s}=0$. Finally, passing to the limit in \eqref{11}, we get that $(\underline{y},\underline{z}, \underline{k})$ is a solution of the RBDSDE \eqref{1}.

\medskip
Let $(y,z,k)$ be any solution of the RBDSDE \eqref{1}. By virtue of Theorem \ref{o}, we have $y^{n}\leq y$, for all $n\geq0$ and therefore, \ $\underline{y}\leq y$ i.e.,  $\underline{y}$ is the minimal solution.
\end{proof}
\begin{remark}
We can prove the maximal solution result for BDSDEs \eqref{1} when the coefficient $f$ is right-continuous and decreasing.
\end{remark}

\label{lastpage-01}

\begin{thebibliography}{99}

\bibitem{Bah}
Bahlali\text{, K.}, Hassani\text{, M.}, Mansouri\text{, B.}, and Mrhardy\text{,
  N.}
\newblock One barrier reflected backward doubly differential equations with
  continuous generator.
\newblock {\em C.R. Acad. Sci. Paris, Ser. I 347 (2009)}, pages
  \text{}1201--1206.

\bibitem{LSm}
Lepeltier\text{, J.P} and San Martin\text{, J.}
\newblock Backward stochastic differential equations with continuous
  coefficients.
\newblock {\em Statist. Probab. Lett}, 32\text{\ }(4):\ \text{}425--430, 1997.

\bibitem{NO}
N'zi\text{, M.} and Owo\text{, J.-M.}
\newblock Backward doubly stochastic differential equations with discontinuous
  coefficients.
\newblock {\em Statist. Probab. Lett}, 79:\ \text{}920--926, 2009.
\newblock doi:10.1016/j.spl.2008.11.011.

\bibitem{PardPeng}
Pardoux\text{, E.} and Peng\text{, S.}
\newblock Backward doubly stochastic differential equations and systèmes of
  quasilinear \textrm{SPDEs}.
\newblock {\em Probab. Theory Related Fields.}, 98:\ \text{}209--227, 1994.

\bibitem{Saisho}
Saisho\text{, Y.}
\newblock \textrm{SDE} for multidimensional domains with reflecting boundary.
\newblock {\em Probab. Theory Related Fields}, 74\text{\ }:\ \text{}455--477,
1987.

\bibitem{Shal} Yufen\text{, S.}; Yanling\text{, G.} and Kai\text{, L.}
\newblock Comparison theorem of backward doubly stochastic
differential equations and application.
\newblock {\em Stoch. Anal. Appl.}, 23: no.1, 97--110, 2005.
\end{thebibliography}
\end{document}